\documentclass[11pt, oneside, letterpaper]{amsart} 

\usepackage[margin=1.15in, left=1.2in, right=1.2in]{geometry}
\usepackage{amsmath,amsfonts,amssymb,amscd,amsthm,amsbsy,epsf,graphics} 
\usepackage{stmaryrd}
\usepackage{setspace}
\usepackage{comment}
\usepackage[dvipsnames]{xcolor}
\usepackage[colorlinks,linkcolor=blue,citecolor=blue]{hyperref} 
\usepackage{epsfig} 
\usepackage{subcaption}
\usepackage{tikz-cd}
\usepackage{bm}
\usepackage{graphicx}  
\usepackage{lipsum}
\usepackage{enumitem} 

\usepackage{todonotes} 

\usepackage[square,numbers]{natbib}

\makeatletter
\def\l@subsection{\@tocline{2}{0pt}{4pc}{5pc}{}}
\let\oldtocsection=\tocsection
\let\oldtocsubsection=\tocsubsection
\let\oldtocsubsubsection=\tocsubsubsection
 
\renewcommand{\tocsection}[2]{\hspace{0em}\oldtocsection{#1}{#2}}
\renewcommand{\tocsubsection}[2]{\hspace{0em}\oldtocsubsection{#1}{#2}}
\renewcommand{\tocsubsubsection}[2]{\hspace{2em}\oldtocsubsubsection{#1}{#2}}

\newtheorem{theorem}{Theorem}[section] 
\newtheorem*{theorem*}{Theorem}
\newtheorem{corollary}[theorem]{Corollary}
\newtheorem*{corollary*}{Corollary}
\newtheorem{lemma}[theorem]{Lemma}
\newtheorem{proposition}[theorem]{Proposition}
\newtheorem*{proposition*}{Proposition}
\newtheorem{conjecture}[theorem]{Conjecture}
\newtheorem{claim}[theorem]{Claim}
\newtheorem{question}[theorem]{Question}
\newtheorem*{question*}{Question}

\newtheorem*{problem*}{Problem}

\theoremstyle{definition}
\newtheorem{definition}[theorem]{Definition}

\newtheorem{remark}[theorem]{Remark} 
\newtheorem*{remark*}{Remark}

\newtheorem*{acknowledgement*}{Acknowledgements}

\theoremstyle{plain}

\newtheoremstyle{cases}
  {6pt plus 3 pt}
  {2pt}
  {\bfseries}   
  {}
  {\bfseries}
  {.}
  {.5em}
  {}
\theoremstyle{cases}

\numberwithin{subcase}{case} 
\numberwithin{subsubcase}{subcase}
\numberwithin{equation}{subsection} 

\graphicspath{{figures/}} 

\title{Toroidal $3$-manifolds have circularly orderable fundamental groups}

\author[Steven Boyer]{Steven Boyer}
\thanks{Steven Boyer was partially supported by NSERC grant RGPIN 2024-06196}
\address{D\'epartement de Math\'ematiques, Universit\'e du Qu\'ebec \`a Montr\'eal, 201 President Kennedy Avenue, Montr\'eal, Qc., Canada H2X 3Y7.}
\email{boyer.steven@uqam.ca}
\urladdr{http://www.cirget.uqam.ca/boyer/boyer.html}

\author[Cameron McA. Gordon]{Cameron McA. Gordon} 
\address{Department of Mathematics, University of Texas at Austin, 1 University Station, Austin, TX 78712, USA.}
\email{gordon@math.utexas.edu}

\author[Ying Hu]{Ying Hu}
\thanks{Ying Hu was partially supported by NSF grant DMS-2409398}
\address{School of Mathematical Sciences, Shanghai Jiao Tong University, 800 Dongchuan Rd, Minhang District, Shanghai, 200240, China}
\email{yinghu@sjtu.edu.cn}
\urladdr{https://yinghu-math.github.io}



\begin{document}

\begin{abstract}
In this article we prove that toroidal $3$-manifolds have circularly-orderable fundamental groups by showing that they admit finite cyclic covers with left-orderable fundamental groups. These covers are also not $L$-spaces and often admit co-orientable taut foliations, as predicted by the $L$-space conjecture. As a consequence, we verify a conjecture of Ba and Clay, which characterises graph manifolds with circularly-orderable fundamental groups. 
\end{abstract}

\maketitle

\section{Introduction}

A group is left-orderable ($LO$) if it is non-trivial and admits a strict total order that is invariant under left multiplication. We say that a $3$-manifold is $LO$ if its fundamental group is left-orderable. 

Finite groups are readily seen to be non-left-orderable, as are many infinite groups, including the fundamental groups of certain $3$-manifolds. Hyperbolic examples of non-$LO$ manifolds can be obtained by taking double branched covers of non-split alternating links \cite{BGW13}, while Sol manifold rational homology $3$-spheres provide simple examples of toroidal  manifolds which are non-$LO$ \cite{BRW05}. 

A major reason for interest in the property $LO$ for $3$-manifolds is the $L$-space conjecture, which asserts that an orientable, compact, connected, irreducible $3$-manifold is $LO$ if and only if it is not an $L$-space (is $NLS$), and if and only if it admits a co-orientable taut foliation (is $CTF$) \cite{BGW13, Juhasz2015}. If $M$ has positive first Betti number then it has all three properties (\cite{BRW05}, \cite{Gab83}). In particular this holds if $M$ has non-empty boundary (and is not the $3$-ball), so we restrict our attention to closed $3$-manifolds. It will also be convenient for us to restrict our attention to manifolds with non-trivial fundamental groups (i.e. $M \not \cong S^3$). We assume this below.

Any closed $3$-manifold with infinite fundamental group has a finite cover with positive first Betti number; this is classical for Seifert fibred manifolds, due to Luecke for toroidal manifolds \cite{Lu88}, and to Agol for hyperbolic manifolds \cite{Agol13}. Hence all orientable, irreducible $3$-manifolds with infinite fundamental groups are virtually $LO$, though in general there is little understanding of which finite covers realize this property. 

In this article, we study when irreducible $3$-manifolds with infinite fundamental group admit finite cyclic covers that are $LO$. Not only are finite cyclic covers among the easiest covers to study, but, as we will discuss below, the existence of an $LO$ finite cyclic cover is equivalent to such a $3$-manifold group being circularly orderable. 

Intuitively, one can think of a circular order on a non-trivial countable group as an embedding of the group elements into the circle $S^1$ such that left multiplication in the group extends to an orientation preserving homeomorphism of the circle (see Definition \ref{def: co}). It is not hard to see that a left-order on a group gives rise to a circular order (\S \ref{sec: background}). Hence, a left-orderable group is circularly-orderable. It is also known that a finite group is circularly-orderable if and only if it is cyclic (see \cite[Proposition 2.5]{CG21} for instance).

We call a $3$-manifold $CO$ if its fundamental group admits a circular order. In \cite{BaCl24}, Ba and Clay showed that an orientable, irreducible $3$-manifold $M$ with infinite fundamental group is $CO$ if and only if $M$ admits a finite cyclic cover that is $LO$ \cite[Theorem 1.1]{BaCl24}. 
They also conjectured that graph manifolds with infinite fundamental groups are $CO$ \cite[Conjecture 6.11]{BaCl24} and verified the truth of this statement except for several infinite families. Here, a {\it graph manifold} is an orientable $3$-manifold which is either Seifert fibred or an irreducible $3$-manifold all of whose JSJ pieces are Seifert fibred.

In this article, we show that, in fact, all toroidal $3$-manifolds are $CO$.

\begin{theorem}
\label{thm: tor means co}
A closed, orientable, irreducible, toroidal $3$-manifold is $CO$. 
\end{theorem}

As a corollary of Theorem \ref{thm: tor means co}, we obtain the following characterisation of $CO$ graph manifolds, which resolves the conjecture of Ba and Clay mentioned above. 

\begin{corollary}
    \label{cor: graph manifolds}
A graph manifold is $CO$ if and only if it is not a Seifert fibred manifold with finite, non-cyclic fundamental group.
\end{corollary}

We prove Theorem~\ref{thm: tor means co} by constructing explicit finite cyclic $LO$ covers of toroidal $3$-manifolds. See Theorem \ref{thm: cocyclic nls and lo} and Remark \ref{rem: orders of the finite cover} below.

Combining \cite[Theorem 1.1]{BaCl24} with the $L$-space conjecture led Ba and Clay to formulate the following assertion. 

\begin{conjecture}[Conjecture 3.4 in \cite{BaCl24}]
If $W$ is an irreducible, rational homology $3$-sphere that is not a lens space, then the following are equivalent:
\vspace{-.1cm} 
\begin{enumerate}
\setlength\itemsep{0.3em}
\item[{\rm (1)}]  The fundamental group of $W$ is circularly orderable.
\item[{\rm (2)}]   There exists a finite cyclic cover $\widetilde{W}$ of $W$ that supports a co-orientable taut foliation.
\item[{\rm (3)}]   There exists a finite cyclic cover $\widetilde{W}$ of $W$ that is not an $L$-space.
\end{enumerate}
\end{conjecture}

Our main contribution addresses this conjecture in the context of toroidal manifolds.

A {\it knot manifold} is a compact, connected, orientable, irreducible $3$-manifold with boundary an incompressible torus. A {\it fibred knot manifold} is a knot manifold which fibres over the circle with compact, orientable fibre having connected boundary.

\begin{theorem}
\label{thm: cocyclic nls and lo}
An irreducible, toroidal rational homology $3$-sphere $W$ admits a finite cyclic cover $\widetilde W \to W$ such that $\widetilde W$ is $LO$ and $NLS$. Further, $\widetilde W$ can also be chosen to be $CTF$ if $W$ contains an essential torus which either bounds a fibred knot manifold or does not bound an integer homology solid torus.  
\end{theorem}

\begin{remark}[The cyclic covers]
\label{rem: orders of the finite cover}
  The two finite cyclic covers used in the proof of Theorem \ref{thm: cocyclic nls and lo} are constructed in \S\ref{sec: constructions}. It is interesting to note that one of them, constructed  in Proposition \ref{prop: ls}, is a cover $\widetilde W_0 \to W$ of degree $ls$ where,  if we express $W$ as a union of two knot manifolds $W= M_1\cup M_2$ along their boundaries, $s$ is the distance between the rational longitudes of $M_1$ and $M_2$ and $l$ is the least common multiple of their orders in $H_1(M_1)$ and $H_1(M_2)$. Claims \ref{claim: nls, lo, ctf} and \ref{claim: nls} show that $\widetilde W_0$ is $NLS$, and so the $L$-space conjecture contends that it is also $LO$ and $CTF$, though as we explain in \S \ref{sec: proofs}, we cannot verify this in all cases. To complete the proof of Theorem \ref{thm: cocyclic nls and lo},  we construct another finite cyclic cover of positive first Betti number when $H_1(M_1)$ and $H_1(M_2)$ both have non-trivial torsion (see Proposition \ref{prop: t_i ne 0}), and this cover allows us to deal  with the remaining cases.  
    \end{remark}

 We conclude the introduction with a discussion of a possible topological characterisation of $CO$ manifolds. We restrict our attention to rational homology $3$-spheres as otherwise the manifold is both $LO$ and $CTF$ (\cite{BRW05, Gab83}) and therefore $CO$ and laminar. 
 
One consequence of the $L$-space conjecture would be that a closed, orientable, irreducible rational homology $3$-sphere is $LO$ if and only if it is $CTF$. This leads one to consider the possibility of an analogous characterisation of $CO$ rational homology $3$-spheres. In one direction, several conditions are known to imply $CO$: admitting a pseudo-Anosov flow, a co-orientable taut foliation, or a tight essential lamination with solid torus guts. Since each of these conditions implies that the manifold is laminar, i.e. contains an essential lamination, one wonders if laminar implies $CO$. Since laminar implies having infinite fundamental group, Corollary \ref{cor: graph manifolds} shows that laminar implies $CO$ for Seifert fibre spaces. Theorem \ref{thm: tor means co} shows that this is also true for toroidal manifolds, which are also laminar. This leaves the hyperbolic case. Also see the discussion after Question 6.3 in \cite{Calegari-problems-2003}.

\begin{question}
\label{qn: laminar imply co?}
Does laminar imply $CO$ for hyperbolic rational homology $3$-spheres?
\end{question} 

One difficulty here is that there are not many non-$CO$ hyperbolic $3$-manifolds known. In fact, to our knowledge, the only published example is the Weeks manifold [CD03], and it is not known whether or not it is laminar.

The following is a special case of Question \ref{qn: laminar imply co?}.

\begin{question} 
\label{qn: haken imply co?}
Are Haken hyperbolic rational homology $3$-spheres $CO$?
\end{question} 

Let $W$ be a Haken rational homology $3$-sphere. Then $W$ contains an incompressible surface which separates $W$ into $M_1$ and $M_2$, say. If the torsion subgroups of $H_1(M_1)$ and $H_1(M_2)$ are non-zero, then the proof of Proposition \ref{prop: t_i ne 0} shows that $W$ is $CO$. So in Question \ref{qn: haken imply co?} it is enough to consider the case where $M_1$ or $M_2$ is an integral homology handlebody.

Whether $CO$ implies laminar has a negative answer for Seifert fibre spaces. The obvious examples are the lens spaces. For more interesting examples, let $W$ be a Seifert fibre space whose base orbifold is the $2$-sphere with three cone points that is either Euclidean or hyperbolic. Then $\pi_1(W)$ is infinite, and hence is circularly orderable by Corollary \ref{cor: graph manifolds}. On the other hand, by Corollary 4 of \cite{Brittenham93}, if $W$ is laminar, then $W$ admits a horizontal foliation. Therefore, if the Seifert invariants of $W$ do not satisfy the conditions of \cite{JN85, Nai94}, then $W$ has no horizontal foliation. Therefore, $W$ is $CO$ and non-laminar.

Theorem \ref{thm: tor means co} (trivially) settles the toroidal case. Thus we are again left with the hyperbolic case.

\begin{question}
\label{qn: co imply laminar?}
Does $CO$ imply laminar for hyperbolic rational homology $3$-spheres?
\end{question}

The only non-laminar hyperbolic $3$-manifolds known are due to Fenley \cite{Fen07}, obtained by Dehn filling certain once-punctured torus bundles.

\begin{question} 
\label{qn: fenley's co?}
Are Fenley's non-laminar hyperbolic $3$-manifolds $CO$?
\end{question} 

If the answer to Question \ref{qn: fenley's co?} is yes, then $CO$ does not imply laminar for hyperbolic rational homology $3$-spheres. On the other hand, if the answer is no, one would find the first infinite family of non-$CO$ hyperbolic manifolds.

\subsection*{Organisation} Background material on circularly ordered groups, slope detection, and gluing of knot manifolds is contained in \S \ref{sec: background}. In \S \ref{sec: constructions}  we construct the two finite cyclic covers of a toroidal $3$-manifold used in the proof of Theorem  \ref{thm: cocyclic nls and lo}. Finally, we prove Theorem \ref{thm: tor means co}, Corollary \ref{cor: graph manifolds} and Theorem \ref{thm: cocyclic nls and lo} in \S \ref{sec: proofs}. 

\subsection*{Acknowledgements} The authors would like to thank Adam Clay for helpful discussions on the theory of circularly ordered groups.

\section{Background material}
\label{sec: background}

\subsection{Circularly orderable groups}
\begin{definition}
\label{def: co}
A {\it circular order} of a non-trivial group $G$ is a map $c :G^3 \to \{\pm 1,0\}$
satisfying:
\begin{enumerate}
\setlength\itemsep{0.3em}
    \item If $(g_1,g_2,g_3)\in G^3$, then $c(g_1,g_2,g_3)=0$
    if and only if $\{g_1,g_2,g_3\}$ are not all distinct;
    \item For all $g_1,g_2,g_3,g_4\in G$, we have
    \[
    c(g_1,g_2,g_3)-c(g_1,g_2,g_4)
    +c(g_1,g_3,g_4)-c(g_2,g_3,g_4)=0;
    \]
    \item For all $g,g_1,g_2,g_3\in G$, we have
 $c(g_1,g_2,g_3)=c(gg_1,gg_2,gg_3)$.
\end{enumerate}
A group is called {\it circularly orderable}, $CO$ for short, if it admits a circular order. Every left-order can be viewed as a circular order. Intuitively, this is similar to obtaining $S^1$ from $\mathbb{R}$ by adding one point at infinity. More precisely, given a left-order $<$ on $G$, one can define a circular order $c: G^3 \to \{\pm 1,0\}$ by setting $c(g_1,g_2,g_3)=\operatorname{sign}(\sigma)$ when $\{g_1,g_2,g_3\}$ are distinct and $g_{\sigma(1)} < g_{\sigma(2)} < g_{\sigma(3)}$, and defining $c(g_1,g_2,g_3)=0$ otherwise.
\end{definition}

The following classical proposition will be used frequently below. For a proof, see \cite[Proposition 2.3(2)]{BaCl24}.

\begin{proposition}
\label{prop: lexicographic}
If there is a short exact sequence $1 \to K \to G \to Q \to 1$ where $K$ is $LO$ and $Q$ is $CO$, then $G$ is $CO$.
\end{proposition}

The group $\mbox{Homeo}_+(S^1)$ is $CO$, and countable $CO$ groups can be characterised dynamically as precisely those countable groups isomorphic to non-trivial subgroups of $\mbox{Homeo}_+(S^1)$. For infinite $3$-manifold groups, a stronger result holds.

\begin{proposition}
\label{prop: inf image}
If $W$ is an orientable, irreducible $3$-manifold and there is a homomorphism $\rho: \pi_1(W) \to \mbox{Homeo}_+(S^1)$ with infinite image, then $W$ is $CO$. 
\end{proposition}

\begin{proof}
A non-trivial  subgroup of infinite index of the fundamental group of an orientable, irreducible $3$-manifold has positive first Betti number (see the proof of \cite[Proposition 3.2]{BRW05}), so the kernel of $\rho$ is either trivial or left-orderable. The conclusion then follows from Proposition \ref{prop: lexicographic}. 
\end{proof}

\subsection{Slope detection and gluing}

The proof of Theorem \ref{thm: cocyclic nls and lo} uses the slope detection and gluing techniques developed in \cite{BC17}, \cite{HRW1} and, in particular, \cite{BGH26}. Roughly speaking, certain slopes on the boundary of a knot manifold are singled out (i.e. ``detected") using Heegaard Floer homology ($NLS$-detection), left-orders ($LO$-detection), or foliations ($CTF$-detection). We direct the reader to \cite{BGH26} for precise definitions. The importance of slope detection can be seen in the fact that a closed $3$-manifold $W$ will have property $* = NLS, LO$ or $CTF$ if it can be expressed as a union of two knot manifolds $W = M_1 \cup M_2$ in such a way that the gluing map matches a $*$-detected slope on $\partial M_1$ with a $*$-detected slope on $\partial M_2$. Multislopes on the boundary of manifolds with multiple boundary components, all tori, can also be $LO$-, $NLS$-, and $CTF$-detected, and there are corresponding generalisations of the gluing theorems. As these results will be invoked formally in our arguments, we content ourselves with stating those parts of the theory that we need and refer the reader to \cite{BGH26} for the details. 

The detection result we will use is a special case of Theorem 1.3 of \cite{BGH26}.

\begin{proposition} [\cite{BGH26}]
\label{prop: meridional detn} 
Let $M \not \cong S^1 \times D^2$ be an irreducible rational homology solid torus whose longitude  $\lambda_M$ is integrally null-homologous. 
\begin{enumerate}[leftmargin=*] 
\setlength\itemsep{0.5em}
\item[{\rm (1)}]  Each  rational slope of distance $1$ from $\lambda_M$ is NLS-detected. 
\item[{\rm (2)}] If $M$ is an integer homology solid torus, then each rational slope of distance $1$ from $\lambda_M$ is $LO$-detected.  
\item[{\rm (3)}] If $M$ fibres over the circle, then each rational slope of distance $1$   from 
$\lambda_M$ is $CTF$-detected.
\end{enumerate}
\end{proposition}
It is expected that the conclusions of parts (2) and (3) of this theorem hold without their added assumptions.

Finally, the gluing theorem we need is a special case of Theorem 7.6 of \cite{BGH26}.

\begin{proposition}[\cite{BGH26}] 
\label{prop: general * gluing}
Let $W = M_0 \cup M_1 \cup \cdots \cup M_n$ be an irreducible rational homology $3$-sphere expressed as a union of submanifolds $M_0, M_1, M_2, \ldots, M_{n}$. Here, $M_1, M_2, \ldots, M_{n}$ are knot manifolds and $M_0$ is a compact, connected, orientable $3$-manifold whose boundary consists of $n$ incompressible tori $T_1, T_2, \ldots , T_n$, where $M_0 \cap M_i = T_i$ for $1 \leq i \leq n$. Fix $\ast \in \{NLS, LO, CTF\}$ and suppose that $(\alpha_1, \alpha_2, \ldots, \alpha_n)$ is a $\ast$-detected rational multislope on $\partial M_0$. If $\alpha_i$ is $\ast$-detected in $M_i$ for $1 \leq i \leq n$, then $W$ has property $\ast$. 
\end{proposition}

\section{Construction of the cyclic covers}
\label{sec: constructions}

In this section we take $W$ to be an irreducible rational homology $3$-sphere which can be expressed as the union of two knot manifolds $M_1, M_2$ along their boundaries: 
$$W = M_1 \cup M_2.$$ 
Our goal is to construct the two finite cyclic covers of $W$ used in the proof of Theorem  \ref{thm: cocyclic nls and lo}. Here is our first construction.

\begin{proposition}
\label{prop: t_i ne 0}
Suppose that $W = M_1 \cup_T M_2$ is an irreducible toroidal rational homology $3$-sphere, where $M_1, M_2$ are knot manifolds with boundary $T$. If the torsion subgroups of $H_1(M_1)$ and $H_1(M_2)$ are non-zero, then there is a finite cyclic cover $\widetilde W \to W$, where $\widetilde W$ has positive first Betti number. In particular, $\widetilde W$ is $NLS, LO$, and $CTF$.  
\end{proposition}

\begin{proof}
As  $W$ is a rational homology $3$-sphere, $M_1$ and $M_2$ are rational homology solid tori. Then $H_1(M_i) \cong \mathbb Z \oplus T_1(M_i)$, where $T_1(M_i)$ is the torsion subgroup of $H_1(M_i)$, and $H_2(M_i) = 0$. Hence, 
$$H_1(W, T) \cong H_1(M_1, T) \oplus H_1(M_2, T) \cong H^2(M_1) \oplus H^2(M_2) \cong T_1(M_1) \oplus T_1(M_2)$$
Thus, if there are an integer $n \geq 2$ and homomorphisms $\varphi_1: T_1(M_1) \to \mathbb Z/n$ and $\varphi_2: T_1(M_2) \to \mathbb Z/n$, we obtain a homomorphism $H_1(W) \to \mathbb Z/n$ via the composition $H_1(W) \to H_1(W, T) \cong T_1(M_1) \oplus T_1(M_2) \xrightarrow{\; \varphi_1 + \varphi_2 \;} \mathbb Z/n$. Note that any such homomorphism vanishes on the image of $H_1(T)$ in $H_1(W)$.

First suppose that $T_1(M_1)$ and $T_1(M_2)$ have cyclic summands with coprime orders $r_1 \geq 2$ and $r_2 \geq 2$ respectively. Applying the construction above to the compositions $T_1(M_1) \twoheadrightarrow \mathbb Z/r_1 \rightarrowtail \mathbb Z/r_1r_2$ and $T_1(M_2) \twoheadrightarrow \mathbb Z/r_2 \rightarrowtail \mathbb Z/r_1r_2$, we obtain an epimorphism $H_1(W) \twoheadrightarrow \mathbb Z/r_1r_2$. The associated cover $\widetilde W$ splits into the inverse images $\widetilde M_1$ of $M_1$ and $\widetilde M_2$ of $M_2$ with $r_2$ and $r_1$ components respectively. Further, the inverse image $\widetilde T$ of $T$ has $r_1r_2$ components, so consideration of the end terms of the Mayer-Vietoris sequence in reduced homology of the union $\widetilde W = \widetilde M_1 \cup_{\widetilde  T} \widetilde  M_2$, 
$$H_1(\widetilde W) \xrightarrow{\; \partial \;} \widetilde{H}_0(\widetilde  T) \to \widetilde{H}_0(\widetilde M_1) \oplus \widetilde{H}_0(\widetilde   M_2) \to 0,$$
shows that $b_1(\widetilde W) \geq \mbox{rank}(\partial) = \mbox{rank}(\widetilde{H}_0(\widetilde  T)) -(\mbox{rank}(\widetilde{H}_0(\widetilde M_1)) +  \mbox{rank}(\widetilde{H}_0(\widetilde M_2))) = (r_1r_2 - 1) - ((r_1-1) + (r_2-1)) = (r_1-1)(r_2-1) \geq 2$.

 On the other hand, if $T_1(M_1)$ and $T_1(M_2)$ do not have cyclic summands with coprime orders, there are an integer $n \geq 2$ and epimorphisms $\varphi_1: T_1(M_1) \twoheadrightarrow \mathbb Z/n$ and $\varphi_2: T_1(M_2) \twoheadrightarrow \mathbb Z/n$ which piece together to give an epimorphism $H_1(W) \twoheadrightarrow \mathbb Z/n$. Then $\widetilde M_1, \widetilde M_2$, and $\widetilde T$, as above, have one, one, and $n$ components respectively. Arguing as in the previous case then shows that $b_1(\widetilde W)\geq 1$, which completes the proof. 
\end{proof}

For our second construction, let $\lambda_i \in H_1(T)$ denote the rational longitude of $M_i$ and let $d_i \geq 1$ be its order in $H_1(M_i)$. Set
$$\bar d_1 = d_1/\gcd(d_1, d_2), \; \bar d_2 = d_2/\gcd(d_1, d_2)$$
so that if $l = \mbox{lcm}(d_1, d_2)$, then  
$$l = \bar d_1 d_2 = d_1 \bar d_2$$ 
Finally, set 
$$s = \Delta(\lambda_1, \lambda_2),$$
where $\Delta(\cdot, \cdot)$ denotes the distance between rational slopes. Since $W$ is a rational homology $3$-sphere, $s \geq 1$. 

\begin{proposition}
\label{prop: ls}
Suppose that $W = M_1 \cup_T M_2$ is an irreducible toroidal rational homology $3$-sphere, where $M_1, M_2$ are knot manifolds with boundary $T$. Then there is a degree $ls$ cyclic cover $p: \widetilde W_0 \to W$ with the following properties $:$

$(1)$ The inverse images of $M_1, M_2$, and $T$ in $\widetilde W_0$ have $\bar d_2, \bar d_1$, and $l$ components respectively.
    
$(2)$ If $d_1, d_2 > 1$, then $b_1(\widetilde W_0) \geq 1$ and therefore $\widetilde W_0$ is $LO, NLS$ and $CTF$.
    
$(3)$ If $\widetilde T_j$ is a component of $p^{-1}(T)$, then $\lambda_1, \lambda_2$ lift to classes $\tilde \lambda_1, \tilde \lambda_2 \in H_1(\widetilde T_j)$ such that $\Delta(\tilde \lambda_1, \tilde \lambda_2) = 1$.

\end{proposition}

\begin{proof} 
There are properly embedded (connected) essential surfaces $F_1, F_2$ in $M_1, M_2$ whose boundaries have $d_1, d_2$ components and represent $d_1 \lambda_1, d_2 \lambda_2 \in H_1(T)$. Since each $F_i$ is non-separating, there is a class $\alpha_i \in H_1(M_i)$ such that $\alpha_i \cdot [F_i] = 1$ for $i = 1, 2$.

Choose a dual class $\mu_1$ to $\lambda_1$ on $T$. Since $\lambda_2$ is primitive and $\Delta(\lambda_1, \lambda_2) = s$, (up to sign) we can write $\lambda_2 = a\lambda_1 + s\mu_1$, where $\gcd(a, s) = 1$. Hence, up to replacing $\mu_1$ by $\mu_1 + k \lambda_1$ for some integer $k$, we can suppose that $\gcd(a, ls) = 1$.

Consider the 2-chain $C = a\bar d_2F_1 - \bar d_1F_2$ in $W$ whose boundary is
$$a\bar d_2d_1 \lambda_1 - \bar d_1d_2 \lambda_2 = al \lambda_1 - l(a\lambda_1 + s\mu_1) = -ls \mu_1$$

Thus $C$ is a (mod $ls$) 2-cycle, so represents a class $\xi \in H_2(W; \mathbb Z/ls)$. Its Poincar\'e dual in $H^1(W; \mathbb Z/ls)$ can be thought of as a homomorphism 
$f: H_1(W) \to \mathbb Z/ls$ represented topologically as the (mod $ls$) intersection with $\xi$:
$$f(\alpha) \equiv \alpha \cdot \xi \; \mbox{ (mod $ls$)}$$
By construction,
\vspace{-.2cm}
\begin{enumerate}

\item[{\rm (a)}] $f$ vanishes on the images of $H_1(F_1)$ and $H_1(F_2)$ in $H_1(W)$. In particular, 
$$f(\lambda_1) \equiv f(\lambda_2) \equiv 0 \mbox{ (mod $ls$)};$$

\item[{\rm (b)}] $f(\alpha_1) \equiv a \bar d_2$ and  $f(\alpha_2) \equiv - \bar d_1$ (mod $ls$); 

\vspace{.2cm} \item[{\rm (c)}] $f(\mu_1) \equiv al$ (mod $ls$). 

\end{enumerate}

Hence as $\gcd(a, ls) = 1$ and $\gcd(\bar d_1, \bar d_2) = 1$, we have $\gcd(\bar d_1, a\bar d_2, ls) = 1$. It then follows from (b) above that $f$ is an epimorphism. Let
$\widetilde W_0 \to W$ be the associated degree $ls$ cyclic cover.  

As each $M_i$ is a rational homology solid torus, $H_1(M_i)$ decomposes as $\mathbb Z \oplus T_1(M_i)$, where $\alpha_i$ generates the $\mathbb Z$ factor and $T_1(M_i)$ is torsion. By the topological description of $f$, the value of $f$ on the class of a $1$-cycle lying in $M_1$ is $a \bar d_2$ times its intersection with $F_1$, and its value on the class of a $1$-cycle lying in $M_2$ is $-\bar d_1$ times its intersection with $F_2$. As such, it vanishes on the images of $T_1(M_1)$ and $T_1(M_2)$. Thus, by (b) above,
$$f(\mbox{image}(H_1(M_1) \to H_1(W))) =\langle a \bar d_2 \rangle =\langle  \bar d_2 \rangle\leq \mathbb Z/ls$$
and
$$f(\mbox{image}(H_1(M_2) \to H_1(W))) = \langle \bar d_1 \rangle \leq \mathbb Z/ls $$
Since $\lambda_1$ and $\mu_1$ generate $H_1(T)$, (a) and (c) imply that 
$$f(\mbox{image}(H_1(T) \to H_1(W))) = \langle al \rangle = \langle l \rangle \leq \mathbb Z/ls$$
It follows that if $\widetilde T, \widetilde M_1, \widetilde M_2$ are the inverse images of $T, M_1, M_2$ in $\widetilde W_0$, then $\widetilde T$ has $l$ components, $\widetilde M_1$ has $\bar d_2$ components, and $\widetilde M_2$ has $\bar d_1$ components. This proves assertion (1) of the proposition.

For (2), assume that $d_1, d_2 > 1$. Then the identities $l = d_1\bar d_2 = \bar d_1 d_2$ show that $\bar d_1, \bar d_2 \leq l/2$. Now use the end terms of the Mayer-Vietoris sequence in reduced homology of the union $\widetilde W_0 = \widetilde M_1 \cup_{\widetilde  T} \widetilde  M_2$, as in the proof of Proposition \ref{prop: t_i ne 0}, to see that $b_1(\widetilde W_0) \geq (l + 1) - (\bar d_1 + \bar d_2) \geq 1$, thus proving (2).  

Finally consider (3). Represent $\lambda_1, \lambda_2$ by simple closed curves $c_1, c_2 \subset T$ which intersect transversely in $s$ points. By (1), we can index the components of $\widetilde  T$ by $\widetilde  T_1, \widetilde  T_2, \ldots , \widetilde  T_l$. Then the restriction of $p$ to each $\widetilde T_j$ is a degree $s$ cover $\widetilde T_j \to T$, so that the $s$ points of intersection between $c_1$ and $c_2$ lift to $s^2$ points of intersection between their inverse images in $\widetilde T_j$. By (a) above, $c_i$ lifts to $s$ disjoint copies of a simple closed curve on $\widetilde T_j$, each representing a class $\tilde \lambda_i$. A simple counting argument then shows that $\Delta(\tilde \lambda_1, \tilde \lambda_2) = 1$. This completes the proof. 
\end{proof}

\section{Proofs of the main results}
\label{sec: proofs} 

We begin by deducing Theorem \ref{thm: tor means co} and Corollary \ref{cor: graph manifolds} from Theorem \ref{thm: cocyclic nls and lo}.

\begin{proof}[Proof of Theorem \ref{thm: tor means co}]
Let $W$ be a closed, orientable, irreducible, toroidal $3$-manifold. Theorem \ref{thm: cocyclic nls and lo} implies that there is a finite degree cyclic cover $\widetilde W \to W$ with $\widetilde W$ $LO$. Then there are a positive integer $n$ and short exact sequence $1 \to \pi_1(\widetilde W) \to \pi_1(W)  \to \mathbb Z/n \to 1$, which implies that $W$ is $CO$ (Proposition \ref{prop: lexicographic}). 
\end{proof}

\begin{proof}[Proof of Corollary \ref{cor: graph manifolds}]
The only reducible graph manifolds are $S^1 \times S^2$ and $P^3 \# P^3$, and both are $CO$; this is obvious in the first case and follows from the exact sequence $1 \to \mathbb Z \to \mathbb Z/2 * \mathbb Z/2 \to \mathbb Z/2 \to 1$ in the second. Thus we can assume that $W$ is irreducible and, as above, that $W$ is a rational homology $3$-sphere. Theorem \ref{thm: tor means co} deals with the case that $W$ is toroidal. Otherwise, $W$ is an irreducible Seifert fibre space without essential tori. As such, it admits a Seifert structure whose base orbifold $\mathcal{B}$ is the $2$-sphere with at most three cone points. 

Since a finite group is $CO$ if and only if it is cyclic, we may suppose that $\pi_1(W)$ is infinite. Then $\mathcal{B}$ is either Euclidean or hyperbolic. In the first case, it  is either $S^2(3,3,3), S^2(2,4,4)$, or $S^2(2, 3, 6)$, and the reader will verify that for $n = 3, 4, 6$ respectively, there is a surjection $\pi_1(\mathcal{B}) \to \mathbb Z/n$ which is injective on finite cyclic subgroups. Hence, there is a $n$-fold cyclic cover $S^1 \times S^1 \to \mathcal{B}$ which lifts to an $n$-fold cyclic cover $\widetilde W \to W$, where $\widetilde W$, being a circle bundle over the torus, is $LO$. Thus $W$ is CO. 

On the other hand, if $\mathcal{B}$ is hyperbolic, its fundamental group is isomorphic to a subgroup of $PSL(2, \mathbb R) \cong \mbox{Isom}_+(\mathbb H^2)$ and the composition $\pi_1(W) \twoheadrightarrow \pi_1(\mathcal{B}) \rightarrowtail  PSL(2, \mathbb R) \rightarrowtail  \mbox{Homeo}_+(S^1)$ has infinite image, so $W$ is $CO$ by Proposition \ref{prop: inf image}. 
\end{proof}

The remainder of the section is devoted to the proof of Theorem \ref{thm: cocyclic nls and lo}. 

\begin{proof}[Proof of Theorem \ref{thm: cocyclic nls and lo}]
Assume that $W = M_1 \cup_T M_2$ is an irreducible toroidal rational homology $3$-sphere, where $M_1, M_2$ are knot manifolds with boundary $T$. We will use the notation of Proposition \ref{prop: ls}.  Thus $d_i$ is the order of the rational longitude $\lambda_i$ of $M_i$, $l$ is the least common multiple of $d_1$ and $d_2$, and $s$ is the distance between $\lambda_1$ and $\lambda_2$. Since $W$ is a rational homology 3-sphere, $s = \Delta(\lambda_1, \lambda_2) \geq 1$. Reindex, if necessary, so that $d_1 \leq d_2$.

Let $p: \widetilde W_0 \to W$ be the degree $ls$ cover of Proposition \ref{prop: ls}. 

\begin{claim}
\label{claim: nls, lo, ctf}
If $d_1 > 1$, then $\widetilde W_0$ is $NLS, LO$, and $CTF$.
\end{claim}

\begin{proof} 
This is Proposition \ref{prop: ls}(2). 
\end{proof} 

Suppose that $d_1 = 1$, and therefore $d_2 = l$. Proposition \ref{prop: ls}(1) then implies that the inverse images $\widetilde M_1, \widetilde M_2$, and $\widetilde T$ of $M_1, M_2$, and $T$ in $\widetilde W_0$ have $d_2, 1$, and $d_2$ components, respectively. Then each component $\widetilde M_{1j}$ of $\widetilde M_1$, $1 \leq j \leq l$, is a knot manifold. If some $\widetilde M_{1j}$ is not a rational homology solid torus, then $b_1(\widetilde W_0) > 0$, so that $\widetilde W_0$ is $NLS, LO,$ and $CTF$.  Assume then that each $\widetilde M_{1j}$ is a rational homology solid torus. 

Recall the compact, orientable, connected surface $F_i$ properly embedded in $M_i$ whose boundary is $d_i \lambda_i$. It follows from (a) of the proof of Proposition \ref{prop: ls} that $F_1$ and $F_2$ lift to $\widetilde W_0$. Hence, as $d_1 = 1$, the rational longitude $\tilde \lambda_{1j}$ of $\widetilde M_{1j}$ is integrally null-homologous. If $\tilde \lambda_{2j}$ is the slope of a lift of $\lambda_2$ to $\partial \widetilde M_{1j}$, then Proposition \ref{prop: ls}(3) implies that $\Delta(\tilde \lambda_{1j}, \tilde \lambda_{2j}) = 1$. 

\begin{claim}
\label{claim: nls}
If $d_1 = 1$, then $\widetilde W_0$ is $NLS$.
\end{claim}

\begin{proof} 
Each lift of $F_2$ to the connected manifold $\widetilde M_2$ is non-separating, so the first Betti number of the manifold obtained by attaching $l$ twisted $I$-bundles $N_1, N_2, \ldots, N_l$ over the Klein bottle to $\widetilde M_2$, where the rational longitude of $N_j$ is identified with $\tilde \lambda_{2j}$, is strictly positive. Hence, the multislope $(\tilde \lambda_{21}, \tilde \lambda_{22}, \ldots, \tilde \lambda_{2l})$ is $NLS$-detected in $\widetilde M_2$ (\cite[Definition 4.6]{BGH26}). Since $\Delta(\tilde \lambda_{1j}, \tilde \lambda_{2j}) = 1$, Proposition \ref{prop: meridional detn}(1) shows that $\tilde \lambda_{2j}$ is $NLS$-detected in $\widetilde M_{1j}$ and therefore Proposition \ref{prop: general * gluing} implies that $\widetilde W_0$ is $NLS$.  
\end{proof}

Our approach to understanding when $\widetilde W_0$ is $LO$ or $CTF$ when $d_1 = 1$ is similar. The first thing to note is that it can be proven that $(\tilde \lambda_{21}, \tilde \lambda_{22}, \ldots, \tilde \lambda_{2l})$ is $LO$-detected in $\widetilde M_2$ as in \cite[Example 6.3]{BGH26}, and it is $CTF$-detected in $\widetilde M_2$ by the main result of \cite{Gab83}. Proposition \ref{prop: general * gluing} then implies that to show that $\widetilde W_0$ is $LO$, respectively $CTF$, it suffices for $\tilde \lambda_{2j}$ to be $LO$-detected, respectively $CTF$-detected, in $\widetilde M_{1j}$ for each $j$. Though we cannot do this in all cases (see the restrictions on  $LO$ and $CTF$ slope detection in parts (2) and (3) of Proposition \ref{prop: meridional detn}), we can in the situations considered in Claims \ref{claim: ihst} and \ref{claim: fibred} below. 

\begin{claim}
\label{claim: ihst}
If $d_1 = 1$ and one of $M_1, M_2$ is an integer homology solid torus, then $\widetilde W_0$ is $LO$. 
\end{claim}

\begin{proof} 
Without loss of generality, we can assume that $M_1$ is an integer homology solid torus. Since $\Delta(\tilde \lambda_{1j}, \tilde \lambda_{2j}) = 1$, the following lemma implies that $\tilde \lambda_{2j}$ is $LO$-detected in $\widetilde M_{1j}$. As noted above, this implies that $\widetilde W_0$ is $LO$.
\end{proof}

\begin{lemma}
\label{lemma: lo detd}
Suppose that $M$ is an irreducible integer homology solid torus with incompressible boundary and $\widetilde M \to M$ is the cyclic cover of degree $n \geq 1$. Then any slope of distance $1$ to the rational longitude of $\widetilde M$ is $LO$-detected in $\widetilde M$. 
\end{lemma}

\begin{proof}
Note that $\widetilde M$ is a knot manifold. If $b_1(\widetilde M) > 1$, then any slope on $\partial \widetilde M$ is $*$-detected, where $* = LO, NLS$, or $CTF$; see \cite[Corollaries 6.12, 4.5, and 5.4]{BGH26}. Hence we may assume that $\widetilde M$ is a rational homology solid torus. Note that the rational longitude $\tilde \lambda$ of $\widetilde M$ is integrally null-homologous. 

By Theorem 1.5(1) of \cite{BGH26}, there exists a homomorphism $\rho: \pi_1(M) \rightarrow \mbox{{\rm Homeo}}_+(S^1)$ whose restriction to $\pi_1(\partial M)$ has a fixed point on $S^1$. The fact that $M$ is an integer homology solid torus implies that $H^2(M) \cong 0$, and therefore the Euler class $e(\rho) \in H^2(M)$ is zero. It follows that $\rho$ lifts to a homomorphism $\tilde \rho: \pi_1(M) \to \mbox{{\rm Homeo}}_{\mathbb Z}(\mathbb R) = \{f \in \mbox{Homeo}_+(\mathbb R) \; | \; f(x + 1) = f(x) + 1\} \leq \mbox{Homeo}_+(\mathbb R)$. Theorem 1.5(2) of \cite{BGH26} then implies that if $\lambda \in \pi_1(\partial M) \leq \pi_1(M)$ represents the rational longitude of $M$, then the translation number of $\tilde \rho(\lambda)$ is $2g(M) - 1 > 0$, where $g(M)$ is the minimal genus of a compact, connected, orientable, essential surface with connected boundary properly embedded in $M$. 

Consider the restriction $\rho_0$ of $\rho$ to $\pi_1(\widetilde M) \leq \pi_1(M)$. Clearly, $\rho_0(\pi_1(\partial \widetilde M))$ has a fixed point on $S^1$. The restriction $\tilde \rho_0$ of $\tilde \rho$ to $\pi_1(\widetilde M)$ is a lift of $\rho_0$ for which the translation number of $\tilde \rho_0(\tilde \lambda)$ equals the translation number of $\tilde \rho(\lambda) = 2g(M) - 1 > 0$. By \cite[Proposition 6.10]{BGH26}, any slope of distance $1$ to $\tilde \lambda$ is $LO$-detected in $\widetilde M$. 
\end{proof}

\begin{claim}
\label{claim: fibred}
If $d_1 = 1$ and one of $M_1, M_2$ is a fibred knot manifold, then $\widetilde W_0$ is $CTF$. 
\end{claim}

\begin{proof} 
Without loss of generality, we can assume that $M_1$ is a fibred knot manifold. Then each $\widetilde M_{1j}$ is as well and therefore Proposition \ref{prop: meridional detn}(3) implies that $\tilde \lambda_{2j}$ is $CTF$-detected in $\widetilde M_{1j}$ for each $j$. As noted above, this implies that $\widetilde W_0$ is $CTF$.
\end{proof}

\begin{claim}
\label{claim: no ihst}
If $d_1 = 1$ and neither $M_1$ nor $M_2$ is an integer homology solid torus, then $W$ has a finite cyclic cover which is $LO, NLS$, and $CTF$. 
\end{claim}

\begin{proof} 
This is Proposition \ref{prop: t_i ne 0}. 
\end{proof}

The first assertion of Theorem \ref{thm: cocyclic nls and lo} follows from Claims \ref{claim: nls, lo, ctf}, \ref{claim: nls}, \ref{claim: ihst},  
and \ref{claim: no ihst}. For the second part of the theorem, if $d_1>1$, the conclusion follows from Claim~\ref{claim: nls, lo, ctf}. Claim~\ref{claim: no ihst} handles the case where $d_1=1$ and neither $M_1$ nor $M_2$ is an integer homology solid torus. Finally, suppose $d_1=1$, and one of the $M_i$ is an integer homology solid torus. If we also have one of the $M_i$ is a fibred knot manifold, then Claims~\ref{claim: nls}, \ref{claim: ihst}, and \ref{claim: fibred} imply that $\widetilde{W}_0$ is $NLS$, $LO$ and $CTF$, respectively.

This completes the proof of Theorem \ref{thm: cocyclic nls and lo}
\end{proof}

{
\footnotesize
\bibliographystyle{alpha}
\bibliography{bgh_co}
}

\end{document}